\documentclass[a4paper]{amsart}
\usepackage[latin9]{inputenc}
\usepackage{amsfonts}
\usepackage{amssymb}
\usepackage{amsmath}
\usepackage{amsthm}
\usepackage{geometry}
\usepackage{verbatim}
\usepackage[toc,page]{appendix}
\usepackage{mathrsfs}
\usepackage[square, numbers, comma]{natbib}

\usepackage{enumerate}
\usepackage{amsbsy}

\usepackage[all]{xy}
\usepackage{pb-diagram,pb-xy}
   \bibpunct{[}{]}{,}{n}{}{,}
\input cyracc.def

\begin{document}
\providecommand{\keywords}[1]{\textbf{\textit{Keywords: }} #1}
\newtheorem{theorem}{Theorem}[section]
\newtheorem{lemma}[theorem]{Lemma}
\newtheorem{prop}[theorem]{Proposition}
\newtheorem{kor}[theorem]{Corollary}
\theoremstyle{definition}
\newtheorem{defi}{Definition}[section]
\theoremstyle{remark}
\newtheorem{remark}{Remark}
\newtheorem{problem}{Problem}
\newtheorem{question}{Question}
\newtheorem{conjecture}{Conjecture}
\newtheorem{example}{Example}
\newtheorem{condenum}{Condition}

\def\p{\mathfrak{p}}
\def\q{\mathfrak{q}}
\def\s{\mathfrak{S}}
\def\Gal{\mathrm{Gal}}
\def\Ker{\mathrm{Ker}}
\def\Coker{\mathrm{Coker}}
\newcommand{\cc}{{\mathbb{C}}}   
\newcommand{\ff}{{\mathbb{F}}}  
\newcommand{\nn}{{\mathbb{N}}}   
\newcommand{\qq}{{\mathbb{Q}}}  
\newcommand{\rr}{{\mathbb{R}}}   
\newcommand{\zz}{{\mathbb{Z}}}  
\newcommand{\fp}{{\mathfrak{p}}}

\title{On Galois extensions with prescribed decomposition groups}
\author{Kwang-Seob Kim}
\address{Chosun University, Department of Mathematics. Gwangju, South Korea}
\author{Joachim K\"onig}
\address{Korea National University of Education, Cheongju, South Korea}

\begin{abstract}
We study the inverse Galois problem with local conditions. In particular, we ask whether every finite group occurs as the Galois group of a Galois extension of $\mathbb{Q}$ all of whose decomposition groups are cyclic (resp., abelian). This property is known for all solvable groups due to Shafarevich's solution of the inverse Galois problem for those groups. It is however completely open for nonsolvable groups. In this paper, we provide general criteria to attack such questions via specialization of function field extensions, and in particular give the first infinite families of Galois realizations with only cyclic decomposition groups and with nonsolvable Galois group. We also investigate the analogous problem over global function fields.
\end{abstract}
\maketitle
%
%

\section{Introduction}
The inverse Galois problem (over a number field $K$) asks whether every finite group $G$ occurs as a Galois group over $K$. While the problem is way out of reach at this point, heuristics have been proposed (e.g., by Malle and Bhargava) which imply that the answer should not simply be ``yes", but rather ``yes, and for systematic reasons". Such systematic reasons come from prescribing the local behavior of a Galois extension at a certain set $S$ of prime ideals of $K$, and then asking whether there exist Galois extensions of $K$ with group $G$ (for short: $G$-extensions) $F/K$ fulfilling these local conditions. A famous version of such an ``Inverse Galois problem with local conditions" is Grunwald's problem, which corresponds exactly to the case where $S$ is a fixed finite set of primes.

A different class of problem arises when $S$ is exactly the set of ramified primes of $F/K$. This leads to problems such as a strong ``Inverse Galois problem with prescribed inertia and/or decomposition groups". We present this problem in general form in the following question:

\begin{question}
\label{ques:decgps}
Let $G$ be a finite group,  $K$ a number field 
 and $I_1,\dots, I_r$ a collection of cyclic subgroups of $G$ such that $G$ equals the normal closure of  $\langle I_1,\dots, I_r \rangle$. For each $i=1,\dots, r$, let $D_i$ be a subgroup of $G$ with $I_i \trianglelefteq D_i$ and $D_i/I_i$ cyclic. Do there exist (tame) $G$-extensions  $L/K$, such that for each prime $p$ of $K$ ramified in $L$, the pair of inertia group and decomposition group at $p$ is conjugate to $(I_i,D_i)$ for some $i\in \{1,\dots, r\}$?
\end{question}

This paper aims at investigating some particularly interesting cases of Question \ref{ques:decgps}, presented in Questions \ref{ques:cyclic} and \ref{ques:loc_ab} below.

\begin{defi}
\label{def:loccyc}
Let $K$ be a number field. Call a Galois extension $L/K$ {\it locally cyclic} if all completions $L\cdot K_p/K_p$ at primes $p$ of $K$ are cyclic.
\end{defi}

 Since unramified extensions of $K_p$ are automatically cyclic, Definition \ref{def:loccyc} is about the decomposition groups at ramified primes of $L/K$, so as a special instance of Question \ref{ques:decgps}, we are lead to the following:

\begin{question}
\label{ques:cyclic}
Let $G$ be a finite group and $K$ a number field. Are there infinitely many, pairwise linearly disjoint, locally cyclic Galois extensions of $K$ with Galois group $G$?
\end{question}

Note that we include the condition of linear disjointness in Question \ref{ques:cyclic} in order not to trivialize the answer for certain small groups. Indeed, if one locally cyclic extension $L/K$ with group $G$ is known, then it is an easy task to construct infinitely many $G\times C_2$-extensions of $K$, all locally cyclic, and all containing the subextension $L/K$, simply by varying the quadratic extension; those extensions however should not really count as ``essentially new".

The following property is weaker, but somewhat more flexible than the locally cyclic property.
\begin{defi}
Let $K$ be a number field. Call a Galois extension $L/K$ {\it locally abelian} if all completions $L\cdot K_p/K_p$ at primes $p$ of $K$ are abelian.
\end{defi}

Note that the compositum of locally abelian extensions is again locally abelian (the analog is not true for locally cyclic extensions!). There is therefore a maximal locally abelian extension $K^{loc-ab}$ of $K$ (for locally cyclic, this does not hold). 

\begin{question}
\label{ques:loc_ab}
Let $G$ be a finite group and $K$ a number field.  Are there infinitely many, pairwise linearly disjoint, locally abelian Galois extensions of $K$ with Galois group $G$?
In particular, is $G$ a quotient of $\Gal(K^{loc-ab}/K)$?
\end{question}

A positive answer to Question \ref{ques:cyclic} (and a fortiori, the weaker Question \ref{ques:loc_ab}) is known for all finite solvable groups due to the nature of Shafarevich's solution of the inverse Galois problem for those groups, see Theorem \ref{thm:shaf}. We are not aware of any nonsolvable group for which Question \ref{ques:cyclic} has previously been answered.

Among the various authors which have studied locally cyclic and/or locally abelian extensions (in various contexts) are Ohtani (\cite{Ohtani}), Nomura (\cite{Nomura}), Caputo and Vinatier (\cite{CV}).

\subsection*{Relevance of locally cyclic and locally abelian Galois extensions}
\label{sec:relevance}
The questions about existence of locally cyclic, resp., locally abelian Galois extensions with prescribed Galois group are of interest by themselves, as special instances of an inverse Galois problem with prescribed local behavior. However, they are also relevant to certain other number-theoretical problems. We present a selection of such applications.

\subsubsection*{Embedding problems}

Locally cyclic extensions are quite useful when it comes to solving central embedding problems.
This is not a formal statement, and indeed, to guarantee solvability of such embedding problems, one requires the still stronger notion of a {\it Scholz extension} (cf.\ \cite[Chapter IV.10]{MM}). See however, e.g., Theorem \ref{thm:centr_ext} for a sample application using locally cyclic extensions (with certain extra conditions) with Galois group $S_5$ 
to construct such extensions also for certain nonsplit central extensions of that group. For an introduction into embedding problems, see Section \ref{sec:embed_basics}.

\subsubsection*{Unramified extensions over small number fields}

In general, the construction problem of unramified Galois extensions $K/k$ of number fields with given Galois groups is not so difficult if we impose no restrictions (on the degree of $k$, on the maximality of $K/k$, etc.), see the construction of Uchida (\cite{Uchida}) and Yamamoto (\cite{Yamamoto}). However, if we require the degree of $k$ to be ``small" for ``big" given Galois group $\Gal(K/k)$, the problem turns out to be highly difficult.  As will be explained in Section \ref{sec:unram}, if $G$ is the Galois group of a tame locally abelian extension $F/\mathbb{Q}$, then $G$ occurs as the Galois group of an unramified extension of some cyclic number field. One also obtains a relatively small base field. See Lemma \ref{lem:unram} for the precise result.

Indeed, it is a folklore conjecture that unramified $G$-extensions exist even over suitable {\it quadratic} number fields, which is however out of reach in general.

\subsubsection*{Weak approximation}
Let $V$ be a variety over a number field $K$. Recall that $V$ is said to satisfy {\it weak approximation} if for any finite set $S$ of primes of $K$, the set $V(K)$ of $K$-rational points is dense within the set $\prod_{p\in S} V(K_p)$, where $K_p$ denotes the completion of $K$ at $p$.
Locally cyclic extensions are known to occur naturally in relation to questions about weak approximation on certain varieties. In particular, if $K/k$ is a locally cyclic Galois extension of number fields, then the associated norm-one torus $R^1_{K/k}(\mathbb{G}_m)$ is known to satisfy weak approximation, due to, e.g., \cite[Corollary 2]{Vos}.
See also \cite[Lemma 6.7]{Frei15}, where it is shown that for abelian extensions $K/k$ of number fields, the converse implication holds as well.
%
%

\subsubsection*{Intersective polynomials}
Another specific application of locally cyclic extensions is to the problem of so-called {\it intersective polynomials}, that is, polynomials $f\in O_K[X]$ (necessarily reducible!) with Galois group $G$ which possess a root in every completion of $K$, but not in $K$. It is easy to see that intersective polynomials exist for every group which occurs as a Galois group (apart from cyclic groups of prime power order, for which there is a trivial group-theoretical obstruction), cf.\ \cite{Sonn09}. However, there is a group-theoretical lower bound for the number of irreducible factors of such a polynomial, namely the minimum number $m$ such that there exist proper subgroups $U_1,\dots, U_m$ of $G$ with $\bigcup_{x\in G, 1\le i\le m} U_i^x = G$ and  $\bigcap_{x\in G, 1\le i\le m} U_i^x = \{1\}$, cf.\ \cite[Proposition 1]{Sonn}. It is in general an open problem whether this lower bound can be obtained. However, from elementary considerations, a $G$-realization in which all decomposition groups are cyclic, immediately yields a positive answer for the latter question, and this approach has been used successfully to obtain such ``optimally intersective" realizations for several classes of groups; see \cite[Theorem 2]{Sonn} for solvable groups, and \cite{BubbSonn}, \cite{Koe18} as well as \cite[\S 4]{KRS} for some constructions of families of minimally intersective polynomials with non-solvable Galois groups. Even though we do not provide new infinite families of groups with optimally intersective realizations in this work, our criteria may well be used to obtain such families in the future.

\section{Structure and main results}
\label{sec:main_res}
The  following treatment is roughly composed of four parts. 
In Section \ref{sec:prereq}, we will study some prerequisites of function fields and their specializations, including Beckmann's result on ramification in specializations;  on simultaneously prime values of certain polynomials (Green-Tao-Ziegler's theorems), and on embedding problems. In Section \ref{sec:locab}, we will deal with the case of locally abelian extensions of $\qq$. 
Here, we provide general criteria (Proposition \ref{thm:inertia2} and Theorem \ref{thm:crit_locab}) suited to attack the problem via specialization of function field extensions. We will then produce several new locally abelian extensions with nonsolvable Galois groups (Corollary \ref{lem:appl_locab}).
In particular, we show:

\begin{theorem}
\label{thm:pgl}
There exist infinitely many primes $p$ such that for each $k\in \mathbb{N}$, the group $PGL_2(p)^k$ possesses locally abelian realizations over $\mathbb{Q}$.
\end{theorem}

Section \ref{sec:unram} discusses in detail the connection of locally abelian extensions with unramified extensions over cyclic number fields introduced above.
In particular, using Theorem \ref{thm:pgl}, we obtain many small cyclic number fields with unramified $PGL_2(p)$-extensions in Corollary \ref{cor:appl_unram}

In Section \ref{sec:loccyc}, we will deal with the case of locally cyclic extensions of $\qq$. Here our main goal is to provide the first examples of non-solvable finite groups with a positive answer to Question \ref{ques:cyclic}. Concretely, we will show (see Theorem \ref{thm:dirprod}):

\begin{theorem}
For any $n\le 10$ and $k\in \mathbb{N}$, there exist infinitely many linearly disjoint locally cyclic extensions of $\qq$ with Galois group $S_n^k$.
\end{theorem}

The goal of Section \ref{sec:embed} is to demonstrate the compatibility of locally cyclic realizations and embedding problems.  We provide some further nontrivial examples of locally cyclic realizations with nonsolvable Galois group equal to a central extension of $S_5$ or $PGL_2(7)$ in Theorem \ref{thm:centr_ext}.
Finally, in Section \ref{sec:app}, we also provide evidence for a general positive answer to Question \ref{ques:cyclic} (and then, a fortiori, to the weaker Question \ref{ques:loc_ab}) by showing an analog over global function fields (Theorem \ref{thm:functionfields}).


{\bf Acknowledgement}:\\
The first author was supported by the research fund from Chosun University, 2020.
The second author was supported by the National Research Foundation of Korea (Grant no. 2019R1C1C1002665).
We also thank Rachel Newton for pointing out the relevance of locally cyclic extensions to weak approximation.

%
%
%
%

\section{Prerequisites}
\label{sec:prereq}

\subsection{Function field extensions}
\label{sec:fctfds}

Let $K$ be a field. A finite Galois extension $F/K(t)$ is called {\it $K$-regular} (or simply {\it regular}, if the base field is clear), if $F\cap \overline{K} = K$. For any $t_0\in K\cup\{\infty\}$ and any place $\mathfrak{p}$ of $F$ extending the $K$-rational place $t\mapsto t_0$, we have a residue field extension $F_{t_0}/K$. This is a Galois extension, not depending on the choice of place $\mathfrak{p}$. We call it the {\it specialization} of $F/K(t)$ at $t_0$.

Now let $K$ be of characteristic zero, and let $F/K(t)$ be a $K$-regular Galois extension with group $G$. 
For each $t_i\in \overline{K}\cup\{\infty\}$, the {\it ramification index} of $F/K(t)$ at $t_i$ is the minimal positive integer $e_i$ such that $F$ embeds into $\overline{K}(((t-t_i)^{1/e_i}))$.\footnote{If $t_i=\infty$, one should replace $t-t_i$ by $1/t$.}  If the ramification index is larger than $1$, then $t_i$ is called a {\it branch point} of $F/K(t)$.
The set of branch points is always a finite set.
If $K$ is of characteristic $0$, then associated to each branch point $t_i$ (of ramification index $e_i$) is a unique conjugacy class $C_i$ of $G$, corresponding to the automorphism $(t-t_i)^{1/e_i}\mapsto \zeta (t-t_i)^{1/e_i}$ of 
$\overline{K}(((t-t_i)^{1/e_i}))$, where $\zeta$ is a primitive $e_i$-th root of unity. The ramification index then equals the order of elements in the class $C_i$. 

\subsection{Inertia and decomposition groups in specializations}
We review some known results about the behavior of inertia and decomposition groups in specializations of function field extensions.

Firstly, the following well-known theorem, cf.\ \cite[Prop.\ 4.2]{Beck}, relates the ramification in specialization of a function field extension to the inertia groups at branch points of that function field extension. 
\begin{theorem}
\label{thm:beck}
Let $K$ be a number field and $N/K(t)$ a $K$-regular Galois extension with Galois group $G$.
Then with the exception of finitely many primes, depending only on $N/K(t)$, the following holds for every prime $\mathfrak{p}$ of $K$.\\
If $t_0\in K$ is not a branch point of $N/K(t)$, then the following condition is necessary for $\mathfrak{p}$ to be ramified in the specialization $N_{t_0}/K$:
 $$\nu_i:=I_{\mathfrak{p}}(t_0, t_i)>0 \text{ for some (automatically unique, up to algebraic conjugates) branch point $t_i$.}$$
 Here $I_{\mathfrak{p}}(t_0,t_i)$ is the intersection multiplicity of $t_0$ and $t_i$ at the prime $\mathfrak{p}$.
Furthermore, $\nu_i>0$ implies that the inertia group of a prime extending $\mathfrak{p}$ in $N_{t_0}/K$ is conjugate in $G$ to $\langle\tau^{\nu_i}\rangle$, where $\tau$ is a generator of an inertia subgroup over the branch point $t\mapsto t_i$ of $N/K(t)$.
\end{theorem}

Regarding the definition of intersection multiplicity $I_{\mathfrak{p}}$ occurring in Theorem \ref{thm:beck}, note that in the special case $K=\mathbb{Q}$ this may be defined conveniently in the following way: Let $f(X)\in \zz[X]$ be the irreducible polynomial of $t_i$ over $\zz$ and $\tilde{f}(X,Y)$ its homogenization (with $\tilde{f}:=Y$ in the special case $t_i=\infty$). Let $t_0=\frac{a}{b}$ with $a,b\in \zz$ coprime, and let $p$ be a prime number. Then $I_p(t_0,t_i)$ is the multiplicity of $p$ in $\tilde{f}(a,b)$.

The following result clarifies the structure of decomposition groups at ramified primes in specializations of function field extensions, cf.\ \cite[Thm.\ 4.1]{KLN}. The key statement, in a nutshell, is that in addition to the structure of the inertia group as given by Theorem \ref{thm:beck}, the structure of the decomposition group at $\mathfrak{p}$ depends on the splitting behavior of $\mathfrak{p}$ in some fixed extension (the residue extension at a branch point of the given regular extension).
\begin{theorem}
\label{thm:kln}
Let $K$ be a number field and $N/K(t)$ a regular Galois extension with Galois group $G$.
Let $t\mapsto a_i \in \mathbb{P}^1(\overline{K})$ be a branch point of $N/K(t)$, and let $I_i$ and $D_i$ denote the inertia and decomposition group at $t\mapsto a_i$ in $N(a_i)/K(a_i,t)$.
Then there exists a finite set $S_0$ of primes of $K$ depending only on $N/K(t)$ such that for all primes $\mathfrak{p}\notin S_0$ and all non-branch points $a\in \mathbb{P}^1(K)$ of $N/K(t)$, the following hold:

\begin{itemize}
\item[a)] Assume $I_{\mathfrak{p}}(a,a_i)>0$, and let $D_{i,\mathfrak{p}'}$ denote the decomposition group at the (unique) prime $\mathfrak{p}'$ of $K(a_i)$ extending $\mathfrak{p} $ such that $I_{\mathfrak{p}'}(a,a_i)>0$. Then the decomposition group $D_{a,\mathfrak{p}}$ at $\mathfrak{p}$ in $N_a/K$ fulfills $\varphi(D_{a,\mathfrak{p}}) = D_{i,\mathfrak{p}'}$, where $\varphi:D_i\to D_i/I_i$ is the canonical epimorphism.
\item[b)] If additionally $I_{\mathfrak{p}}(a,a_i) = 1$, then $D_{a,\mathfrak{p}} = \varphi^{-1}(D_{i,\mathfrak{p}'})$.
\end{itemize}
\end{theorem}

\begin{remark}
\label{rem:combine_beckkln}
Note that Theorem \ref{thm:beck} implies that in case b) of Theorem \ref{thm:kln}, the inertia group at $\mathfrak{p}$ in $N_{a}/K$ is all of $I_i$. Combination of Theorem \ref{thm:beck} and \ref{thm:kln}b) therefore implies that the residue degrees at $\mathfrak{p}$ in $N_{a}/K$ and at $\mathfrak{p}'$ in $N(a_i)_{a_i}/K(a_i)$ are the same.
In particular, if the latter residue degree equals $1$, then the decomposition group at $\mathfrak{p}$ in $N_a/K$ is cyclic and equal to the corresponding inertia subgroup. This implication will be of interest for us in Section \ref{sec:loccyc}.
\end{remark}

While Theorems \ref{thm:beck} and \ref{thm:kln} are primarily designed to clarify the local behavior at ramified primes in specializations, the following result by D\`ebes and Ghazi (\cite[Theorem 1.2]{DebesGhazi}) deals with unramified primes. More precisely, it ensures that prescribed unramified local behaviors can be obtained at prescribed primes, up to suitable choice of the specialization.
\begin{theorem}
\label{thm:debesghazi}
Let $K$ be a number field and $N/K(t)$ a regular Galois extension with Galois group $G$. Then there exists a finite set $S_0$ of primes of $K$ such that for every prime $p$ of $K$ not contained in $S_0$ and every conjugacy class $C$ of $G$, there exist infinitely many $t_0\in K$ such that $N_{t_0}/K$ is Galois with group $G$ and is unramified of Frobenius class $C$ at the prime $p$. More precisely, there is a positive constant $c$ depending only on $N$ and a set $\mathcal{S}$ consisting of at least $N(p)/c$ full residue classes modulo $p$ (where $N(p)$ denotes the cardinality of the residue field of $p$) such that all $t_0\in K$ with $(t_0 \textrm{ mod } p)\in \mathcal{S}$ yield the above property.\footnote{While this bound on the number of admissible residue classes is not directly stated in \cite[Theorem 1.2]{DebesGhazi}, it is evident from the proof, which reduces the problem to a count of mod-$p$ points of a certain curve, and thus to an application of the Lang-Weil bound.}
\end{theorem}

\subsection{Green-Tao-Ziegler's theorems on simultaneously prime values of polynomials}
Due to the above results, obtaining extensions with decomposition groups of a prescribed form among the set of specializations of a given function field extension is essentially reduced to (controlling the behavior of a finite set of ``bad" primes and) forcing the prime divisors of evaluations of certain homogeneous polynomials $\tilde{f}(X,Y)$ (namely, the irreducible polynomials of the branch points) to be simultaneously contained in certain prescribed Chebotarev sets.
This is in general a deep problem. In this section we recall two situations in which unconditional results are available. These will be used to construct locally cyclic extensions in Section \ref{sec:loccyc}.

A landmark result is Tao's and Ziegler's theorem on primes in polynomial progressions, see \cite[Theorem 1.3]{TZ}.
\begin{theorem}[Tao-Ziegler]
\label{thm:greentao}
Let $p_1(Y),\dots, p_k(Y)\in \zz[Y]$ be polynomials with $p_i(0)=0$ for all $i=1,\dots, k$. Let $\mathcal{S}$ be a positive density set of prime numbers. Then there exist infinitely many $(x_0,y_0) \in \zz^2$ such that all of values $x_0+p_i(y_0)$, $i=1,\dots, k$, are primes in $\mathcal{S}$.
\end{theorem}

Note that the case $p_i(Y) = i\cdot Y$ leads to the famous theorem asserting that the primes contain arbitrarily long arithmetic progressions.

A related result is Green-Tao's theorem about simultaneously prime values of affine-linear forms. We give a special case which is sufficient for our purposes. The main theorem of \cite{GT} gives a much stronger result; our application follows in particular due to \cite[Corollary 1.9]{GT}.
\begin{theorem}[Green-Tao]
\label{thm:taoziegler}
Let $f_1(X,Y), \dots, f_k(X,Y) \in \zz[X,Y]$ be affine linear forms in two variables such that 
\begin{itemize}
\item[a)] The product $\prod_{i=1}^r f_i$ has no fixed prime divisor, i.e., for each prime $p$ there exists $(x_0,y_0)\in \mathbb{Z}^2$ such that none of $f_1(x_0,y_0),\dots, f_k(x_0,y_0)$ is divisible by $p$.
\item[b)] The $f_i$ are pairwise affinely independent (i.e., if $a,b,c\in \zz$ such that $af_i + bf_j +c = 0$, then $a=b=c=0$).
\end{itemize}
Then there are infinitely many $x_0,y_0\in \zz$ such that $f_1(x_0,y_0), \dots, f_k(x_0,y_0)$ are simultaneously prime.\footnote{In order to avoid an obstruction coming from the archimedean prime, we count negatives of prime numbers as prime here. This is consistent with the assertion of \cite[Corollary 1.9]{GT} via changing the sign of some of the $f_i$ if necessary.} 
\end{theorem}

\begin{remark}
\label{rmk:gt}
in fact, \cite[Corollary 1.9]{GT} allows to pick the values $(x_0, y_0)$ in $\mathbb{Z}^2 \cap K$ for any prescribed open convex cone $K\subset \mathbb{R}^2$. In particular, the assertion of Theorem \ref{thm:taoziegler} remains valid even with the additional condition that the quotient $x_0/y_0$ should lie in some prescribed real interval of positive length.
\end{remark}

\subsection{Embedding problems}
\label{sec:embed_basics}
In this section, we recall some basic notions from the theory of embedding problems of Galois extensions, which will be relevant in Section \ref{sec:embed}. General studies on embedding problems can be found in Chapter III $\S 5$ and Chapter IV $\S 6$ of \cite{NSW}.

Let $k$ be a number field 
and $G_k$ the absolute Galois group of $k$. Let $K/k$ be a finite Galois extension with the Galois group $G$. For a group extension $1 \arrow{e} A \arrow{e} E \arrow{e,l} {j} G \arrow{e} 1$ of finite groups, the embedding problem $(K/k, \epsilon)$ is defined by the diagram

\[
  \begin{diagram}\dgARROWLENGTH=1.6em
    \node{}
    \node{}
    \node{}
    \node{G_k} \arrow{s,r} {\varphi}
    \node{}\\
    \node{(\epsilon) : 1} \arrow{e}
    \node{A} \arrow{e}
    \node{E} \arrow{e,l} {j}
    \node{G} \arrow{e}
    \node{1}
  \end{diagram}
  \]
where $\varphi$ is the natural restriction map. A continuous homomorphism $\psi$ of $G_k$ to $E$ is called a \textit{solution} of $(K/k, \epsilon)$ if it satisfies the condition $j \circ \psi = \varphi$. When $(K/k, \epsilon)$ has a solution, we call $(K/k, \epsilon)$ is \textit{solvable}. A solution $\psi$ is called a {\it proper solution} if it is surjective. A field $N$ is called a \textit{solution field} (resp. a \textit{proper solution field}) of $(K/k, \epsilon)$ if $N$ is the fixed field of the kernel of a solution (resp. a proper solution). The existence of a proper solution of $(K/k,\epsilon)$ is equivalent to the existence of a Galois extension $M$ of $k$ containing $K$ such that the canonical sequence $1 \to \Gal(M/K) \to \Gal(M/k) \to \Gal(K/k) \to 1$ coincides with $\epsilon$. The embedding problem is called {\it central} if $A\le Z(E)$.
For each prime $\p$ of $k$, we denote by $k_{\p}$ (resp. $K_{\p'})$ the completion of $k$ (resp.\ $K$) by $\p$ (resp.\ an extension $\p'$ of $\p$ to $K$).

Then the {\it local embedding problem} $(K_{\p'}/k_{\p},\epsilon_{\p})$ of $(K/k,\epsilon)$ is defined by the diagram
\[
  \begin{diagram}\dgARROWLENGTH=1.6em
    \node{}
    \node{}
    \node{}
    \node{G_{k_{\p}}} \arrow{s,r} {\varphi|_{\p}}
    \node{}\\
    \node{(\epsilon) : 1} \arrow{e}
    \node{A} \arrow{e}
    \node{E_{\p}} \arrow{e,l} {j_{E_{\p}}}
    \node{G_{\p}} \arrow{e}
    \node{1}
  \end{diagram}
  \]
where $G_{\p}$ is the Galois group of $K_{\p'} / k_{\p}$, which is isomorphic to the decomposition group of $\p$ in $K / k$, $G_{k_{\p}}$ is the absolute Galois group of $k_{\p}$, and $E_{\p}$ is the inverse of $G_{\p}$ by $j$. In the same manner as the case of $(K/k,\epsilon)$, solutions, solution fields
etc.\ are defined for $(K_{\p'}/k_{\p},\epsilon_{\p})$. 

\section{Locally abelian extensions of $\mathbb{Q}$}
\label{sec:locab}
This section is concerned with the construction of locally abelian extensions with prescribed Galois group. We begin with an obvious observation which justifies restricting our attention to $K=\mathbb{Q}$ for most of the following. 
\begin{lemma}
Let $K$ be a number field. Assume that there exist infinitely many linearly disjoint locally abelian (resp., locally cyclic) $G$-extensions over $K$. Then the same holds over any number field $L\supseteq K$.
\end{lemma}

\subsection{General criteria}
Below, we will attack the problem of finding locally abelian
extensions of $\mathbb{Q}$ with prescribed Galois group by specialization of suitable function field extensions. We therefore need to control the local behavior of such specializations.

A convenient way to obtain locally abelian extensions is given by the following elementary observation.
\begin{remark}
\label{rem:order2inertia}
Every tamely ramified extension of $\qq$ in which all ramification indices are $\le 2$ is locally abelian. Indeed, the decomposition groups at ramified primes are then (order $2$) central extensions of cyclic groups, i.e., abelian. Using results from the literature about Galois realizations with inertia groups of order $2$, this shows for example that $S_n$, $A_5$, $PSL_2(7)$, $PSL_2(11)$, $M_{11}$ and several more almost simple groups are quotients of $\Gal(\qq^{loc-ab}/\qq)$. Compare \cite{Ked} for the case $S_n$ and \cite[Proposition 4.5]{KNS} for the other groups mentioned above.
\end{remark}

\begin{prop}
\label{thm:inertia2}
Let $E/\qq(t)$ be a $\qq$-regular $G$-extension such that all non-trivial inertia groups at finite branch points of $E/\mathbb{Q}(t)$ are of order $2$. 
Assume that there exists $t_0\in \qq$, not a branch point of $E/\qq(t)$, such that the specialization $E_{t_0}/\qq$ is locally abelian. 
Then there exist infinitely many pairwise linearly disjoint locally abelian $G$-extensions $F/\mathbb{Q}$. In particular, Question \ref{ques:loc_ab} has a positive answer for all groups $G^d$, $d\in\mathbb{N}$.
Moreover, one may choose the extensions $F/\mathbb{Q}$ such that, for each ramified prime $p$, either the completion $F_p/\mathbb{Q}_p$ equals $(E_{t_0})_p/\mathbb{Q}_p$ or $p$ is tamely ramified of ramification index $2$.
\end{prop}
\begin{proof}
Up to linear transformation in $t$, we may assume $t_0= 0$.
We need to find $t_1\in \mathbb{Q}$, not a branch point of $E/\mathbb{Q}(t)$, such that $\Gal(E_{t_1}/\mathbb{Q})=G$ and all decomposition groups at ramified primes in $E_{t_1}/\mathbb{Q}$ are abelian.
We now choose $t_1\in \zz$.\footnote{This ensures that $t_1$ cannot meet infinity (if that is a branch point) modulo any prime, so that this branch point may be ignored when applying Theorem \ref{thm:beck}.} Then from Theorem \ref{thm:beck}, it follows that all ramified primes in $E_{t_1}/\mathbb{Q}$ outside some finite set $S_0$ (depending only on $E/\mathbb{Q}(t)$) have inertia group of order $2$ - and in particular, are tamely ramified as long as $p\notin S_0\cup \{2\}$. From Remark \ref{rem:order2inertia}, it then follows that all these primes have abelian decomposition group in $E_{t_1}/\mathbb{Q}$.

Now let $S$ be any finite set of prime numbers containing $S_0\cup\{2\}$. Choose $t_1\in \zz$ sufficiently close $p$-adically to $t_0=0$ for all $p\in S$. Then Krasner's lemma yields that the completions of $E_{t_1}/\qq$ and of $E_{t_0}/\mathbb{Q}$ at $p$ are equal, and in particular abelian, for all $p\in S$.\\
Note that these mod-$p$ conditions yield an $S$-adically open set of admissible values $t_1\in \mathbb{Q}$. Due to, e.g., \cite[Proposition 2.1]{PV05}, 
such a set is compatible with Hilbert irreducibility, i.e., among the set of all such specializations $E_{t_1}/\mathbb{Q}$, there are infinitely many pairwise linearly disjoint ones with Galois group $G$, completing the proof.
\end{proof}

The following theorem relaxes the assumptions of Proposition \ref{thm:inertia2}, allowing applications for quite a number of almost-simple groups. Its argument is a variant of Theorem 4.4 in \cite{KNS}.
\begin{theorem}
\label{thm:crit_locab}
Let $E/\qq(t)$ be a $\qq$-regular $G$-extension with branch points $t_1,\dots, t_r$, and let $I_i\le G$ (resp., $e_i\in \nn$) be the inertia group (resp., the ramification index) of $E/\qq(t)$ at $t_i$. Assume that all of the following are fulfilled:
\begin{itemize}
\item[a)] $e_3 = \dots = e_r = 2$.
\item[b)] $e_1\notin \{2,e_2\}$; and either $\gcd(e_1,e_2) = 1$ or $G$ is perfect.\footnote{I.e., $G$ equals its derived subgroup $G'$.}
\item[c)] $N_G(I_1)/I_1$ is abelian (where $N_G(U)$ denotes the normalizer in $G$ of a subgroup $U\le G$), and the extension $1\to I_1\to N_G(I_1)\to N_G(I_1)/I_1 \to 1$ splits.
\end{itemize}
Then there exists a $\mathbb{Q}$-regular $G$-extension $\tilde{E}/\mathbb{Q}(t)$ unramified at $0$, fulfilling the assumptions of Proposition \ref{thm:inertia2}, and such that $\tilde{E}_{0}\subseteq E_{t_1}$ and ${\Gal}(\tilde{E}_0/\mathbb{Q})$ is isomorphic to a subgroup of $N_G(I_1)/I_1$.
%
%
\end{theorem}
\begin{proof} We reduce to Proposition \ref{thm:inertia2} by considering a suitable pullback $E(u)/\qq(u)$, where $\qq(u)/\qq(t)$ is an extension of rational function fields.
Since by assumptions a) and b), $t_1$ is the only branch point of ramification index $e_1$, it follows from the branch cycle lemma (e.g., \cite[Lemma 2.8]{Voe})
that $t_1$ is rational, and we may assume $t_1=0$ without loss. 

We next claim that we may assume $t_2$ to be rational as well. If $e_2>2$, this follows again from the branch cycle lemma, whereas if $e_2=2$, we may simply replace the tuple $(t_1,t_2,\dots,t_r)$ of branch points by $(t_1', \dots, t_{r+1}'):= (t_1, t_2', t_2,\dots, t_r)$, where $t_2'$ is any unramified rational point; we may then work with this new tuple, and have artificially created $e_2=1$. We have thus shown the claim, and may assume without loss of generality that $t_2=\infty$. Consider now the completion at $E/\qq(t)$ at $t\mapsto 0$. This is of the form $K((\sqrt[e_1]{\alpha t}))$, where $K$ is the residue field of $E$ at $t\mapsto 0$ and $\alpha\in K$. Since the extension $1\to I_1\to N_G(I_1)\to N_G(I_1)/I_1 \to 1$ splits, $K((\sqrt[e_1]{\alpha t}))/\mathbb{Q}((t))$ has a totally tamely ramified subextension $M/\mathbb{Q}((t))$ of degree $e_1$. Such an extension is known to always be of the form $M=\mathbb{Q}((\sqrt[e_1]{\pi}))$ 
for some uniformizer $\pi$ of $\mathbb{Q}[[t]]$ (e.g., \cite{Lang}, pp. 52-53), and thus even of the form $M=\mathbb{Q}((\sqrt[e_1]{\alpha' t}))$ for some $\alpha'\in \mathbb{Q}$. We may therefore assume $\alpha = \alpha'\in \qq$. Furthermore the residue extension $K/\qq$ has Galois group $D_1/I_1 \le N_G(I_1)/I_1$ (where $D_1$ denotes the decomposition group at $t_1$), i.e., is abelian.

Now let $u = \sqrt[e_1]{\alpha t}$. Consider the extension $E(u)/\qq(u)$. The extension $\qq(u)/\qq(t)$ is of degree $e_1$, totally ramified at $0$ and $\infty$. By the second part of assumption b), the extensions $E/\qq(t)$ and $\qq(u)/\qq(t)$ are linearly disjoint, even after base change to $\overline{\qq}$. 
(In the case of a perfect group $G$, this is due to the fact that the extension $\qq(u)/\qq(t)$ is solvable, whereas $G$ has no solvable quotient.) 
Therefore $E(u)/\qq(u)$ is a $\qq$-regular $G$-extension. Furthermore, it is unramified at $u\mapsto 0$ due to Abhyankar's lemma, and all non-trivial inertia groups at finite places have order $2$. 

We will now show that residue extension $E(u)_0/\qq$ is a subextension of $K/\mathbb{Q}$. This suffices to prove the assertion, since ${\Gal}(E(u)_0/\mathbb{Q})$ is then a quotient of the abelian group ${\Gal}(K/\mathbb{Q}) \cong D_1/I_1 \le N_G(I_1)/I_1$, and any quotient of an abelian group also occurs as a subgroup.

But note that $E(u)_0$ is certainly contained in the compositum of the completions of $E/\qq(t)$ and $\qq(u)(\zeta_{e_1})/\qq(t)$ at $t\mapsto 0$. The latter completion is $\qq(\zeta_{e_1})((u))/\qq((t))$, and since $\zeta_{e_1}\in K$ (see, e.g., \cite[Lemma 2.3]{KLN}),
 it follows that the compositum of the two completions is still $K((u))/\qq((t))$, meaning that $E(u)_0$ cannot be larger than $K$. This completes the proof
\end{proof}

\subsection{Sample applications}
Below, we give several applications of Proposition \ref{thm:inertia2} and Theorem \ref{thm:crit_locab} to produce locally abelian extensions with almost-simple Galois groups.
\begin{kor}
\label{lem:appl_locab}
Let $G$ be  one of the following:
\begin{itemize}
\item[a)] $G=PGL_2(p)$ where $p$ is a prime modulo which not all of $2$, $3$, $5$ and $7$ are squares.
\item[b)] $G=M$ the sporadic monster group. 
\end{itemize}
Then there are infinitely many pairwise linearly disjoint locally abelian Galois extensions of $\qq$ with group $G$. 
\end{kor}
\begin{proof}
First, let $G=PGL_2(p)$. Then by Corollary 8.10 in \cite{MM}, there is a $\mathbb{Q}$-regular $G$-extension whose ramification type is one of the following (depending on whether $2$, $3$, $5$ or $7$ is a non-square modulo $p$): $(2B, 4A, pA)$, $(2B,6A,pA)$, $(2B,2B,4A,pA)$, or $(2B,2B,3A,pA)$.\footnote{Here, we use the usual ATLAS notation for conjugacy classes.}
Since the normalizer of a cyclic group of order $p$ in $PGL_2(p)$ equals $C_p\rtimes C_{p-1}$,\footnote{Indeed, this normalizer is the image of a Borel subgroup of $GL_2(p)$ under the canonical projection $GL_2(p)\to PGL_2(p)$.}
 all assumptions of Theorem \ref{thm:crit_locab} are fulfilled. The assertion thus follows from  Proposition \ref{thm:inertia2}.

Now let $G=M$.
It is known due to Thompson (\cite{Tho84}) that $G$ has a $\qq$-regular realization with three branch points, and ramification indices $2$, $3$ and $29$ respectively. Let $I$ be a cyclic group of order $29$ in $G$. From the ATLAS of finite simple groups, the normalizer of $I$ in $G$ fulfills $N_G(I)/I \cong C_{28}\times C_3$, and the extension is automatically split by Schur-Zassenhaus since the normal subgroup and its quotient group are of coprime order. Therefore all assumptions of Theorem \ref{thm:crit_locab} are fulfilled, yielding locally abelian extensions of $\mathbb{Q}$ with group $M$.
\end{proof}

It should be noted that Theorem \ref{thm:crit_locab} can be used to generate locally abelian extensions for many more interesting (e.g., simple) finite groups. In particular, \cite{MM} contains many realizations of sporadic simple groups over $\mathbb{Q}(t)$ with exactly $3$ branch points. Quite often, the three ramification indices are distinct and one of them is $2$. It then suffices to check the condition c) of Theorem \ref{thm:crit_locab}; note that this condition is fulfilled automatically e.g. if $I_1$ is self-centralizing in $G$. Then using the results of \cite[Chapter II.9]{MM}, one immediately obtains locally abelian extension for, e.g., the Baby Monster $B$, the Thompson group $Th$, etc. We leave the verification of these claims to the reader.

\subsection{Application to unramified extensions over cyclic number fields}
\label{sec:unram}
Here we give our application of locally abelian extensions to the construction of unramified Galois extensions of cyclic number fields with prescribed Galois group.
We begin with a general observation.
\begin{lemma}
\label{lem:unram}
Let $G$ be the Galois group of a locally abelian extension $F/\mathbb{Q}$. Assume that $F/\mathbb{Q}$ is tamely ramified.
Then $G$ occurs as the Galois group of an unramified extension of some cyclic number field $L$. Moreover, let $m$ denote the least common multiple of all ramification indices at ramified primes in $F/\mathbb{Q}$. Then
 one may choose $L$ such that $[L:\mathbb{Q}]\le m$. 
 \end{lemma}
\begin{proof}
Let $\{p_1,\dots,p_r\}$ be the set of ramified primes of $F/\qq$. Since the completion of $F/\qq$ at any prime extending $p_i$ has abelian Galois group, it follows that $\qq_{p_i}$ must contain the $e_i$-th roots of unity (where $e_i$ denotes the ramification index at $p_i)$, whence there exists a totally ramified $C_{e_i}$-extension of $\qq_{p_i}$. Due to the Grunwald-Wang theorem (e.g., Theorem 9.2.8 in \cite{NSW}), we may choose a cyclic extension $L/\qq$ of degree ${\rm lcm}\{e_1,\dots, e_r\}$ all of whose completions at ramified primes $p_i$ are the prescribed $C_{e_i}$-extensions. Abhyankar's lemma then yields that $FL/L$ is unramified. Since we may easily demand $L$ and $F$ to be linearly disjoint over $\qq$, e.g.\ by imposing further local conditions on $L$, we may assume without loss that $FL/L$ has Galois group $G$. 
\end{proof}

For the special case of solvable groups, the statement and proof of Lemma \ref{lem:unram} occurred in \cite[Main Theorem]{Kim19}.
Moreover, the special case where all ramification indices are equal to $2$ corresponds to unramified extensions of quadratic number fields, as treated, e.g., in \cite{KNS} and \cite{KRS}.

In order to obtain interesting classes of new examples, it is desirable to drop this tameness condition. We can do it at least after replacing the single locally abelian extension $F/\mathbb{Q}$ by a regular extension over $\mathbb{Q}(t)$ with suitable properties.

\begin{lemma}
\label{lem:unram2}
Let $G$ be the Galois group of a $\mathbb{Q}$-regular extension $E/\mathbb{Q}(t)$ which fulfills the assumptions of Theorem \ref{thm:crit_locab}. Assume furthermore that (in the notation of Theorem \ref{thm:crit_locab}) $N_G(I_1)/I_1$ is cyclic of order $n$. Then there exist infinitely many cyclic number fields of degree dividing ${\textrm{lcm}}(2,n)$ possessing unramified $G$-extensions.
\end{lemma}
\begin{proof}
Theorem \ref{thm:crit_locab}, together with Proposition \ref{thm:inertia2}, yields infinitely many locally abelian $G$-extensions $F/\mathbb{Q}$ in which all ramified primes $p$ are either tamely ramified of index $2$, or fulfill $F_p/\mathbb{Q}_p = K_p/\mathbb{Q}_p$, where $K/\mathbb{Q}$ is a fixed cyclic extension of degree dividing $n$ and depending only on $E/\mathbb{Q}(t)$ (namely the extension denoted by $\tilde{E}_0/\mathbb{Q}$ in Theorem \ref{thm:crit_locab}). 

For a fixed $F/\mathbb{Q}$, call the set of primes with the former property $S_1$, and the set of primes with the latter property $S_2$. Note that if $2$ is ramified in $F/\mathbb{Q}$, then $2\in S_2$ by definition.
We may now choose a cyclic extension $L/\mathbb{Q}$ of degree dividing ${\textrm{lcm}}(2,n)$ such that $L_p/\mathbb{Q}_p$ is tamely ramified of index $2$ for all $p\in S_1 \setminus S_2$, and $L_p/\mathbb{Q}_p = K_p/\mathbb{Q}_p$ for all $p\in S_2$. Such $L/\mathbb{Q}$ always exists by the Grunwald-Wang theorem. Indeed, as long as $n$ is odd, our case is never in the exceptional case of the Grunwald-Wang theorem; and if $n = {\textrm{lcm}}(2,n)$ is even, one may choose $[L:\mathbb{Q}]=[K:\mathbb{Q}]$ and simply note that there is no obstruction at the prime $2$, since existence of a degree-$[K:\mathbb{Q}]$ cyclic extension with the required behavior at the prime $2$ is ensured by $K/\mathbb{Q}$ itself. Since $L_p/\mathbb{Q}_p=F_p/\mathbb{Q}_p$ for all $p\in S_2$, one has that $FL/L$ is unramified at all $p\in S_2$; and same at all $p\in S_1\setminus S_2$ by Abhyankar's lemma. Finally, ${\Gal(FL/L)}=G$ is ensured as in the previous proof.
\end{proof} 

\begin{kor}
\label{cor:appl_unram}
Let $G$ be  one of the following:
\begin{itemize}
\item[a)] $G=PGL_2(p)$ where $p$ is a prime modulo which not all of $2$, $3$, $5$ and $7$ are squares.
\item[b)] $G=M$ the sporadic monster group. 
\end{itemize}
Then there are infinitely many cyclic number fields possessing unramified $G$-extensions.
\end{kor}
\begin{proof}
This follows from Lemma \ref{lem:unram2} and the proof of Corollary \ref{lem:appl_locab}. Indeed, in all cases of Corollary \ref{lem:appl_locab}, $N_G(I_1)/I_1$ (in the notation of Theorem \ref{thm:crit_locab}) was cyclic - of order $p-1$ in case a) and of order $84$ in case b).
\end{proof}

\begin{remark}
Following a construction of Uchida (\cite{Uchida}) and Yamamoto (\cite{Yamamoto}), every finite group $G$ can be realized as the Galois group of an unramified Galois extension of {\it some} number field $K$, by embedding $G\le A_n$, and then choosing $K$ as the fixed field of $G$ in a suitable $S_n$-extension $F/\mathbb{Q}$ (namely, with all inertia groups generated by transpositions, which implies that $F/\textrm{Fix}(A_n)$ is unramified). In the case $G=M$, this approach would yield $[K:\mathbb{Q}]$ somewhere near the factorial of $10^{20}$ (which is roughly the smallest permutation degree of $M$), whereas our construction gives $[K:\mathbb{Q}]\le 84$.
\end{remark}

\section{Locally cyclic extensions of $\mathbb{Q}$}
\label{sec:loccyc}
We now move on to the more restrictive notion of locally cyclic extensions. 
A positive answer to Question \ref{ques:cyclic} is known for all solvable groups, due to Shafarevich's solution of the inverse Galois problem for solvable groups (see, e.g., \cite[Chapter IX.6]{NSW}).

\begin{theorem}[Shafarevich]
\label{thm:shaf}
Let $G$ be a finite solvable group and $K$ a number field. 
Then there are infinitely many linearly disjoint\footnote{The linear disjointness is automatic from the assertion for $G^k$ for all exponents $k\in \mathbb{N}$.} $G$-extensions $L/K$ fulfilling the following:
\begin{itemize}
\item[i)] $L/K$ is tamely ramified.
\item[ii)] All decomposition groups at ramified primes in $L/K$ are cyclic and equal the respective inertia groups.
\end{itemize}
In particular, $L/K$ is locally cyclic.
\end{theorem}

More generally, from Shafarevich's method we obtain an answer to Question \ref{ques:cyclic} for certain split group extensions.
\begin{lemma}
\label{rmk:split}
Let $G$ be a finite group and $K$ a number field. Assume that Question \ref{ques:cyclic} has a positive answer for $G$ over $K$. Then it also has a positive answer for all split extensions $G\ltimes N$ with nilpotent kernel $N$.
\end{lemma}
\begin{proof}
This follows from Shafarevich's method of solving split embedding problems with nilpotent kernel (see \cite[Theorem 9.6.7]{NSW}), which, given some $G$-realization $L/K$, guarantees the existence of a $G\ltimes N$-extension $F/K$ containing $L/K$ such that all ramified primes of $L/K$ split completely in $F/L$ and all ramified primes of $F/L$ have cyclic decomposition groups equal to the respective inertia group in $F/K$. 
We apply this result for the split  embedding problem corresponding to $1\to N^k \to (G\ltimes N)^k \to G^k \to 1$, for any positive integer $k$. Since we may begin with a $G^k$-extension $L/K$ which is a compositum of $k$ linearly disjoint locally cyclic $G$-extensions, the resulting $(G\ltimes N)^k$-extension $F/K$ has the property that all decomposition groups in $F/K$ at primes which do not ramify in $L/K$ are cyclic, and all decomposition groups in $F/K$ at primes which ramify in $L/K$ equal the respective decomposition group in $L/K$. The latter decomposition groups are assumed to be such that their projection onto each of the $k$ components is cyclic, whence $F/K$ is the compositum of $k$ linearly disjoint locally cyclic $G\ltimes N$-extensions.
Since this holds for any $k\in \mathbb{N}$, the assertion follows.
\end{proof}


For non-solvable groups, not only is a general answer to Question \ref{ques:cyclic} unknown, but to our knowledge, there is not even a single example of a non-solvable group in the literature with a positive answer to Question \ref{ques:cyclic} (of course, from looking at databases such as {\texttt{lmfdb.org}} or \texttt{http://galoisdb.math.upb.de}, one can easily obtain {\it finitely many} locally cyclic realizations for certain non-solvable groups).

\subsection{A general method}
It is easy to gain specialization criteria in the spirit of Proposition \ref{thm:inertia2}, such as the following (the proof is analogous to that of Proposition \ref{thm:inertia2}, except that instead of Theorem \ref{thm:beck}, one should invoke Theorem \ref{thm:kln}):

(*) Let $E/\mathbb{Q}(t)$ be a $\mathbb{Q}$-regular $G$-extension whose decomposition groups at all the branch points are cyclic and which possesses at least one locally cyclic specialization $E_{t_0}/\mathbb{Q}$ at a non-branch point $t_0$. Then there exist infinitely many locally cyclic $G$-extensions which are specializations of $E/\mathbb{Q}(t)$.

The reason why we will not pursue Criterion (*) any further here is that the assumptions here are much more restrictive than in the previous section and very hard to fulfill for groups of interest. 

Instead, in order to obtain applicable criteria, we will make use of the theorems of Green-Tao-Ziegler on prime specializations of linear forms. Since in these results, the specialization values can in general not be chosen in arbitrary prescribed residue classes, we cannot expect to use the same trick as in the proof of Proposition \ref{thm:inertia2}, dealing with the behavior at certain small primes via Krasner's lemma.
We therefore begin with exhibiting a general method, which however ignores a certain fixed finite set of primes. 

\begin{theorem}
\label{thm:crit}
Let $E/\mathbb{Q}(t)$ be a $\qq$-regular $G$-extension with branch points $t_1, \dots ,t_r$, all rational (i.e., contained in $\mathbb{P}^1(\qq)$). 
Then there exists a finite set $S_0$ of primes of $\qq$ (depending only on $E/\mathbb{Q}(t)$), and infinitely many $G$-extensions of $\qq$ all of whose decomposition groups at ramified primes, except possibly at the ones in $S_0$, 
are cyclic and equal to the respective inertia subgroup.
%
%
\end{theorem}

\begin{proof}
%
%
Up to fractional linear transformation, we can even assume that all branch points $t_1,\dots,t_r$ of $E/\mathbb{Q}(t)$ are integers. 
Let $\mathcal{S}$ be the set of primes which split completely in the compositum $F$ of all residue extensions of $E/\mathbb{Q}(t)$ at branch points $t_i$, $i=1,\dots,r$. 
By increasing $F$ suitably, we may assume that $\mathcal{S}$ is disjoint from the finite set of primes for which the assertion of either Theorem \ref{thm:beck} or Theorem \ref{thm:kln} (for $E/\mathbb{Q}(t)$) fails.
By Chebotarev's density theorem, the set $\mathcal{S}$ is of positive density in the set of prime numbers.
From the Tao-Ziegler theorem about polynomial progressions in primes (Theorem \ref{thm:greentao}), it then follows that there are infinitely many $(x_0:y_0)\in \mathbb{P}^1(\qq)$ such that all values $x_0-t_iy_0$, $i=1,\dots,r$, are primes contained in $\mathcal{S}$. 
From Theorem \ref{thm:beck}, it follows that there exists a finite set $S_0$ of primes (depending only on $E/\mathbb{Q}(t)$) such that all primes $p\notin S_0$ which ramify in $E_{t_0}/\mathbb{Q}$ are divisors of $F(x_0,y_0)$, where $F(X,Y)=\prod_{i=1}^r (X-t_iY)$. Hence all these primes are contained in $\mathcal{S}$.
 Since all these primes have residue degree $1$ in all the residue extensions at branch points of $E/\mathbb{Q}(t)$, Remark \ref{rem:combine_beckkln}
 yields that the decomposition group at the prime $x_0-t_iy_0$ in $E_{t_0}/\mathbb{Q}$ is cyclic and equal to the corresponding inertia group. This completes the proof.
\end{proof}

\begin{remark}
The idea to invoke Green-Tao-Ziegler to obtain results about specializations of regular $G$-extensions occurs in work by Bary-Soroker and Schlank regarding the minimal ramification problem (\cite{BSS}). Note that their application is concerned only with the number of ramified primes (not with their decomposition groups, and hence in particular not with their Frobenius in certain prescribed extensions). What is new in the present work is the combination of Green-Tao-Ziegler with results about decomposition groups in specializations of function fields as exhibited in \cite{KLN}. Note that for our purposes, it would actually be sufficient to force all prime divisors of all values $x-t_iy$ (as in the above proof) into a prescribed Chebotarev set, whereas primality of the values is not required. However, we are not aware of any general criterion other than Green-Tao-Ziegler asserting such a result.
\end{remark}


\subsection{Sample applications: The symmetric groups $S_n$}

To improve on Theorem \ref{thm:crit} at least in special cases by getting rid of the finite exceptional set $S_0$, we look at some extensions of $\mathbb{Q}(t)$ given by explicit polynomials $f(t,X)\in \zz[t,X]$. By doing this, we construct the first infinite families of locally cyclic extensions with nonsolvable Galois groups.

\begin{theorem}
\label{thm:dirprod}
Let $n\le 10$, and let $G=S_n$.
Then for any $k\in \nn$, there exist infinitely many tamely ramified $G^k$-extensions $L/\mathbb{Q}$ such that, 
for all finite primes which ramify in $L/\mathbb{Q}$, the decomposition groups are cyclic and equal the respective inertia groups.
Furthermore, the $L/\mathbb{Q}$ may be chosen pairwise linearly disjoint.
\end{theorem}
\begin{proof}
Since the statement is known from Shafarevich's theorem for solvable groups (see Theorem \ref{thm:shaf}), we may assume $n\ge 5$.
First, let $k=1$.

{\it First step: Choice of a regular extension and its specialization.}
It is well known (and easy to see) that the splitting field of the polynomial $f(t,X) = X^n - t(X-1)$ over $\qq(t)$ yields a $\qq$-regular $S_n$-extension, ramified over $t\mapsto 0$, $t\mapsto \infty$ and one more rational point $t\mapsto \frac{n^n}{(n-1)^{n-1}}$, with inertia groups generated by an $n$-cycle, an $(n-1)$-cycle and a transposition respectively (see, e.g., \cite[p.42]{Serre} for this polynomial (up to linear transformation in $X$)). In fact, the discriminant of $f$ factors as $\Delta(f) = t^{n-1}\cdot((n-1)^{n-1}t - n^n)$ in $\zz[t]$. 
Denote the splitting field of $f$ by $E$. Via linear transformation, we can make all branch points of $E/\mathbb{Q}(t)$ integers, in the spirit of the proof of Theorem \ref{thm:crit}. Concretely, consider the polynomial
$\tilde{f}(t,X) = ab(n^n b+a(n-1)^{n-1})X^n-t(a(X-1)+bX^n)$ (with transcendentals $a,b$), which is the numerator of $f(\frac{at}{ab(n^n b+a(n-1)^{n-1}) - bt}, X)$.
Denote its splitting field by $\tilde{E}$, and note that for any non-branch point $t_0$ of $\tilde{E}/\mathbb{Q}(t)$, we have $\tilde{E}_{t_0}/\qq = E_{\frac{at_0}{ab(n^n b+a(n-1)^{n-1}) - bt_0}}/\mathbb{Q}$ by definition.

 One verifies that the discriminant $\tilde{\Delta}$ of the translated polynomial $\tilde{f}$ with respect to $X$ factors as 
 $$\tilde{\Delta}(t) = a^{n-1}b^{n-2}(n^n b+a(n-1)^{n-1})\cdot t^{n-1}(t-(n-1)^{n-1}a^2 - n^n ab)^{n-2}(t-n^n ab),$$ where the three linear factors in $t$ correspond to the branch points with inertia group $n$-cycle, $(n-1)$-cycle and transposition respectively. 
 
 We denote the factor $ab(n^n b+a(n-1)^{n-1})$ of $\tilde{\Delta}(t)$ by $Q(a,b)$. We fix integer values $a$ and $b$ and homogenize the three linear factors of $\tilde{\Delta}(t)$ to obtain polynomials $f_1(T,S)=T$, $f_2(T,S)=T-((n-1)^{n-1}a^2 - n^n ab)S$ and $f_3(T,S)=T-n^nabS \in \zz[T,S]$. We now claim that, for any choice of integer values $a$ and $b$, 
 there exist infinitely many $t_0\in \qq$ such that all ramified primes of $\tilde{E}_{t_0}/\qq$, apart from possibly the prime divisors of $Q(a,b)$, have cyclic decomposition group equal to the inertia group. Indeed, application of the Tao-Ziegler theorem (Theorem \ref{thm:greentao}) yields infinitely many $t_0=(x_0:y_0)\in \mathbb{P}^1(\mathbb{Q})$ at which all of $f_1,f_2$ and $f_3$ evaluate to primes which are completely split in the compositum of all residue extensions at branch points of $\tilde{E}/\mathbb{Q}(t)$.
Additionally, these primes may be assumed outside any prescribed finite set $S_0$ (in particular, outside of the prime divisors of $Q(a,b)$). As in the proof of Theorem \ref{thm:crit}, the decomposition groups at all these primes in $\tilde{E}_{t_0}/\mathbb{Q}$ are then cyclic and equal to the respective inertia subgroup. Due to Theorem \ref{thm:beck}, this inertia group is generated by an $n$-cycle, an $(n-1)$-cycle and a transposition respectively (for the three different primes). It is an elementary exercise in permutation group theory to show that the Galois group of $\tilde{E}_{t_0}/\mathbb{Q}$ must then be (a primitive permutation group containing a transposition, and hence) $S_n$.

On the other hand, all {\it further} ramified primes of $\tilde{E}_{t_0}/\mathbb{Q}$ must still divide the homogenization of $\tilde{\Delta}$, evaluated at $(x_0,y_0)$,
 and therefore in fact divide $Q(a,b)$. We deal with those primes in the next step.

{\it Second step: Controlling the ``bad" primes.}
Note that the prime divisors of $Q(a,b)$ are in fact bad primes (in the sense of Theorem \ref{thm:beck} and a fortiori Theorem \ref{thm:kln}) for $\tilde{E}/\mathbb{Q}(t)$, but not necessarily for $E/\mathbb{Q}(t)$!
Assume now that we can find distinct prime values $a,b$ such that $n^n b+a(n-1)^{n-1}$ is also a prime, distinct from $a$ and $b$, and such that none of the three primes is exceptional for $E/\mathbb{Q}(t)$ in the sense of Theorem \ref{thm:kln}.
 Since specializing 
 $t\mapsto t_0 = \frac{x_0}{y_0}$ in $\tilde{E}/\mathbb{Q}(t)$ corresponds to specializing $t\mapsto \frac{at_0}{ab(n^n b+a(n-1)^{n-1}) - bt_0}= \frac{a\cdot f_1(x_0)}{-b\cdot f_2(x_0,y_0)} =:\frac{\widehat{x_0}}{\widehat{y_0}}$
 in $E/\mathbb{Q}(t)$,
 Theorem \ref{thm:beck} shows that, for  $t_0\in \qq$ as in the first step, $a$ ramifies in $\tilde{E}_{t_0}$ with $n$-cycle inertia, $b$ with $(n-1)$-cycle inertia and $n^n b+a(n-1)^{n-1}$ with transposition inertia (the latter upon factorizing the expression $(n-1)^{n-1}\widehat{x_0} - n^n \widehat{y_0}$). Furthermore, by Theorem \ref{thm:kln}b), the residue extension at each such prime $p$ equals the residue at $p$ in the geometric residue extension at the respective branch point of $E/\qq(t)$.

But those latter residue extensions are quite explicit. They are $\qq(\zeta_n)$, $\qq(\zeta_{n-1})$ and  the splitting field of $f(\frac{n^n}{(n-1)^{n-1}},X)$. (This follows quickly from considering the normalizers of the respective inertia subgroups in $S_n$; e.g., the normalizer of the $n$-cycle subgroup in $S_n$ is $C_n\rtimes Aut(C_n)$, and since the $n$-th roots of unity must always be contained in the residue field at a ramification-index-$n$ branch point, this leaves only $\qq(\zeta_n)/\qq$ as the residue extension.)\\
One may now use computer search to find values of $a,b$ such that $a$, $b$ and $n^n b+a(n-1)^{n-1}$ are distinct primes, split in the respective residue extensions.
While existence of infinitely many such $(a,b)$ may depend on open conjectures, we easily found {\it some} values for all $n=5,\dots,10$ via computer search.
Concretely, $(n,a,b)\in \{(5,11,13),(6,7,151),(7,113,4027),\\(8, -359,-167),(9,541,3257),(10,21341, 21943)\}$ are possible choices.

It remains to justify that the above prime values are not in the exceptional set $S_0$ of Theorem \ref{thm:kln} for $E/\mathbb{Q}(t)$. This set can be made explicit with some extra effort; in fact, with our particular choice of polynomial $f(t,X)$, one can show that this set only contains prime numbers $\le n$. Since a full proof of this fact would lead us too far, we content ourselves with noting that for the finitely many, explicitly given values of those primes $p$ we treat here, the local behavior could alternatively be checked ad hoc with Magma.

{\footnotesize{\bf Corrigendum (July 2021)}: The values for $n=8$ given above are incorrect (and arose from a computer search in which, falsely, modulo conditions were imposed on $a$ and $b$ rather than on their absolute values). Instead, $a=-17$, $b=5209$ is a correct pair of values for $n=8$.}

{\it Third step: Linear disjointness.}
Finally, it remains to prove that all extensions $\tilde{E}_{t_0}/\qq$ obtained in this way can be assumed linearly disjoint. This is however immediate, since two non-linearly-disjoint $S_n$-extensions would have to have the same quadratic subextension, and this subextension is ramified at the numerator of $t_0$ (in case the $n$-cycle is not in $A_n$) or of $t_0-(n-1)^{n-1}a^2 - n^n ab$ (otherwise), whereas  two extensions obtained via our construction have only $a$, $b$ and $n^n b+a(n-1)^{n-1}$ as jointly ramified primes.
We may therefore assume the sets of ramified primes of all the quadratic subextensions to be distinct, ending the proof for $k=1$.

For arbitrary $k\in \nn$, fix prime numbers $a,b$ as found above, and consider the splitting field $\tilde{E}$ of $\prod_{i=1}^k \tilde{f}(t-c_i, X)$ over $\qq(t)$, with constants $c_1,\dots,c_k\in \zz$. With suitable choice of $c_1,\dots,c_k$, one can of course make the branch point loci of the splitting fields of $\tilde{f}(t-c_i, X)$ (call them $\tilde{E}_i$) pairwise disjoint, whence this defines a $\qq$-regular $G^k$-extension with a total of $3k$ branch points $t\mapsto t_j$, all rational integers. Using once again Theorem \ref{thm:greentao}, there are infinitely many specializations $t\mapsto t_0 = \frac{m_0}{n_0}\in \qq$ such that the values $m_0-t_jn_0$ are simultaneously prime, and all these primes are completely split in the compositum of all residue extensions at branch points of $\tilde{E}/\qq(t)$.
\footnote{Here, we use that Theorem \ref{thm:greentao} is applicable for arbitrary subsets of the primes of positive density. This is not the case for Theorem \ref{thm:taoziegler}, which is the reason why our construction in Theorem \ref{thm:loc_cyclic} in the following section, at least a priori, does not allow a generalization to direct products.} 
By Remark \ref{rem:combine_beckkln}, the primes $m_0-t_jn_0$ ramify in the specialization $\tilde{E}_{t_0}/\qq$ of $\tilde{E}/\qq(t)$, with decomposition group equal to the inertia group. 
 The property $\Gal(\tilde{E}_{t_0}/\qq) = S_{n}^k$ is automatic under these conditions, since the inertia groups at ramified primes obtained in this way include an $n$-cycle, an $(n-1)$-cycle and a transposition inside every single of the $k$ copies of $S_n$, i.e., a set of elements which together generate all of $S_n^k$.

 Restriction to the individual extensions $\tilde{E}_i/\qq(t)$ shows that the only further ramified primes in $\tilde{E}_{t_0}/\qq$ are $p\in \{a, b, n^n b+a(n-1)^{n-1}\}$. If we can ensure that, for all these $p$, the local extensions $(\tilde{E}_i)_{t_0}\cdot \qq_p/\qq_p$ are the same for all $i$, then we are done, since the completion of the compositum of all the $(\tilde{E}_i)_{t_0}/\qq$ is then still the same (cyclic) extension. But note that the primes $a$, $b$, and  $n^n b+a(n-1)^{n-1}$ are fixed independently of $i\in \{1,\dots, k\}$. If we demand that the integers $c_1,\dots, c_k$ are sufficiently close to each other $p$-adically for all $p\in  \{a, b, n^n b+a(n-1)^{n-1}\}$ (which is without loss of generality), then Krasner's lemma yields immediately that the completions of $(\tilde{E}_i)_{t_0}/\qq$ at $p$ are all identical.
This completes the proof.
\end{proof}

A minor blemish in the previous result is the existence of primes which are ramified in all the infinitely many extensions constructed. 
In Section \ref{sec:embed}, we will present a slight variant of the construction which gets rid of this (however, it will not a priori be compatible with taking direct products).

\subsection{Locally cyclic extensions and embedding problems}
\label{sec:embed}

Here we vary our construction of the previous section slightly, to deduce the existence of locally cyclic extensions for certain central extensions of $S_5$. From Lemma \ref{rmk:split} and Theorem \ref{thm:dirprod}, this is automatic for certain split extensions (namely, such with nilpotent kernel). We therefore investigate what can be said about non-split extensions. We begin with a statement essentially regaining the result of Theorem \ref{thm:dirprod} for the special case $G=S_5$, although with certain useful extra conditions, and we also include the group $G=PGL_2(7)$.

\begin{theorem}
\label{thm:loc_cyclic}
Let $G=S_5$ or $G=PGL_2(7)$.
Let $S_1$ be a finite set of prime numbers, and $S_2$ a finite set of ``sufficiently large" prime numbers (i.e., disjoint from some fixed finite set depending only on $G$).
For each $p\in S_2$, pick a conjugacy class $C(p)$ of $G$. Then there exist infinitely many linearly disjoint Galois extensions $K/\qq$ fulfilling each of the following:
\begin{itemize}
\item[a)] $\Gal(K/\mathbb{Q}) \cong G$,
\item[b)] $K/\mathbb{Q}$ is locally cyclic,
\item[c)] All $p\in S_1\cup S_2$ are unramified inside $K$,
\item[d)] For each $p\in S_2$, the Frobenius of $p$ inside $K/\mathbb{Q}$ is in the conjugacy class $C(p)$.
\end{itemize}
\end{theorem}
\begin{proof}
We will construct the extensions $K/\mathbb{Q}$ as specializations of some $\qq$-regular $G$-extension $E/\mathbb{Q}(t)$. We begin by noting that it then suffices to prove statements b), c) and d). Indeed, to obtain that the Galois group of the thus constructed specialization equals the full group $G$, it suffices to enlarge the set $S_2$ such that every conjugacy class of $G$ occurs as $C(p)$ for some $p\in S_2$. The Galois group must then be a subgroup of $G$ intersecting each conjugacy class of $G$ non-trivially, which forces $\Gal(K/\mathbb{Q})\cong G$ by a classical theorem of Jordan (\cite{Jor}).

We begin with the case $G=S_5$.
As in (the case $n=5$ of) the proof of Theorem \ref{thm:dirprod}, consider the polynomial $f(t,X) = X^5 - t(X-1)$ and let $E$ be its splitting field. The homogenized discriminant of $f$ (with respect to $X$) is $T^4 \cdot S^3 \cdot (256 T - 3125 S) \in \zz[T,S]$, and the inertia groups at $t\mapsto 0$, $t\mapsto\infty$ and $t\mapsto 3125/256$ are generated by a $5$-cycle, $4$-cycle and transposition respectively.

We claim that, in order to satisfy conditions b) and c), it suffices to find infinitely many $(t_0:s_0)\in \mathbb{P}^1(\qq)$ such that
\begin{itemize}
\item[i)] $t_0$ is a prime $\equiv 1$ mod $5$.
\item[ii)] $s_0$ is a prime $\equiv 1$ mod $4$.
\item[iii)] $256t_0-3125s_0$ is a prime $\equiv 1$ or $\equiv 3$ mod $8$. 
\end{itemize}
Indeed, using again Theorem \ref{thm:kln}, condition i) ensures that the decomposition group at $t_0$ (in the splitting field of $f(t_0/s_0,X)$) is cyclic of order $5$, ii) ensures that the decomposition group at $s_0$ is cyclic of order $4$, and iii) ensures that $256t_0-3125s_0$ splits completely in $\qq(\sqrt{-2})$ which is the unique quadratic subextension of the ($S_3$)-residue extension at $t\mapsto 3125/256$ in the splitting field of $f(t,X)$. Since the decomposition group at $256t_0-3125s_0$ is contained in the normalizer of $(1,2)$ (up to conjugation), this prime  then has inertia group generated by $(1,2)$ and decomposition group a subgroup of $\langle(1,2), (3,4,5)\rangle \cong C_6$.

Since $256t_0-3125s_0$ is congruent to $1$ mod $8$ as soon as $s_0\equiv 1$ mod $8$, it is sufficient to show that the affine linear forms 
\begin{equation}
\label{eq:forms}
f_1(T) = 5T+1, f_2(S) = 8S+1 \text{ and } f_3(T,S) = 256(5T+1) - 3125(8S+1) 
\end{equation} 
take infinitely many simultaneously prime values (which also can be assumed to be outside of the finite set $S_1\cup S_2$, ensuring assertion c)). Since the product of the three forms has no fixed prime divisor (as, e.g., already seen by evaluating $(T,S)$ to $(0,0)$ and to $(0,1)$), and no two of them are affinely dependent, 
 this follows directly from Green-Tao's theorem (Theorem \ref{thm:taoziegler}).

To ensure that additionally assertion d) holds, we make use of Theorem \ref{thm:debesghazi} for $E/\mathbb{Q}(t)$.
 We then obtain that, for every sufficiently large (depending on $E$) prime $p$ and any class $C$ of $G$, there exists an integer $k(p)$, coprime to $p$, such that every specialization $E_{a_0}/\mathbb{Q}$ with a rational number $a_0=\frac{t_0}{s_0}\equiv k(p)$ mod $p$ has Frobenius class $C$ at the prime $p$. 
Denoting by $R$ the product of all primes in $S_2$ and using the Chinese remainder theorem, it therefore suffices to add the conditions $t_0\equiv k$ mod $R$
and $s_0\equiv 1$ mod $R$, for some prescribed coprime residue $k$ mod $R$, to the above conditions i), ii), iii) and verify that the assumptions of Green-Tao's theorem are still fulfilled. Indeed, we only need to replace $T$ by $RT+k$ and $S$ by $RS+1$ in \eqref{eq:forms}, i.e., show that the three forms 
$\tilde{f_1}(T) := f_1(RT+k) = 5RT + 5k+1$, $\tilde{f_2}(S) := f_2(RS+1) = 8RS+9$ and  $\tilde{f_3}(T,S) = 256\tilde{f_1}(T) - 3125\tilde{f_2}(S)$
 still take infinitely many simultaneously prime values. 

 These three forms are still obviously pairwise affinely independent, and it only remains to show that we can furthermore assume the nonexistence of a fixed prime divisor.
However, since we only introduced mod-$p$ conditions for $p\in S_2$, such a fixed prime divisor could only be an element of $S_2$. Evaluating the product of the three forms at $(0,0)$ gives $\tilde{f_1}(0)\tilde{f_2}(0)\tilde{f_3}(0,0) = 9(5k+1)(256(5k+1)-9\cdot 3125)$ and we only need to ensure that $k$ as above can be chosen such that this expression is coprime to $p\in S_2$.
%
%
%
However, recall that the mod-$p$ residue class of $k$ is the same as the one of $k(p)$. By Theorem \ref{thm:debesghazi}, there is a positive constant $c$ (independent of $p$) such that at least $p/c$ residue classes mod $p$ are admissible, and assuming without loss that $p>5$, one sees that only two of those residue classes make the product $\tilde{f_1}(0)\tilde{f_2}(0)\tilde{f_3}(0,0) = 9(5k+1)(256(5k+1)-9\cdot 3125)$ congruent to $0$ mod $p$. 
This completes the case $G=S_5$.

Now let $G=PGL_2(7)$. Consider the polynomial $f(t,X) = X^7(X+7)-t(X^2+X+7)$. This has Galois group $PGL_2(7)$ over $\mathbb{Q}(t)$ (see, e.g., \cite[Proposition 9]{Koe_Mon} for a polynomial identical to $f$ up to linear transformation)  and 
homogenized discriminant (over $\zz$) $T^6(108T + 7^7S)^3 \cdot S^5$. The branch points are exactly $t\mapsto 0$, $t\mapsto \infty$ and $t\mapsto -\frac{7^7}{108}$, with inertia groups generated by a $7$-cycle, a $6$-cycle and a triple transposition respectively. The corresponding residue fields are $\qq(\zeta_7)$, $\qq(\zeta_6)$ and the splitting field of $X^3+7X^2+49/3 X + 343/6$ (the last splitting field is obtained from evaluating $t\mapsto -\frac{7^7}{108}$ in $f(t,X)$, factoring the resulting polynomial, and noting that the normalizer of a triple transposition in $PGL_2(7)$ has order $12$ (hence the residue extension cannot possibly have degree more than $6$)).
We thus want to find $(t_0:s_0)\in \mathbb{P}^1(\qq)$ such that 
\begin{itemize}
\item[i)] $t_0$ is a prime $\equiv 1$ mod $7$.
\item[ii)] $s_0$ is a prime $\equiv 1$ mod $6$.
\item[iii)] $108t_0+7^7s_0$ is a prime $\equiv 1$ mod $6$. 
\end{itemize}
Note here that condition iii) is because the quadratic subextension of the splitting field of $X^3+7X^2+49/3 X + 343/6$ equals $\qq(\sqrt{-3})$; therefore, as soon as $108t_0+7^7s_0 = p$ is a prime $\equiv 1$ mod $6$, the inertia group at $p$ in the specialization is of order $2$ and the quotient of decomposition by inertia group is of order dividing $3$, implying cyclic decomposition group.

Now one may easily verify that application of Green-Tao's theorem yields locally cyclic extensions with properties c) and d), just like in the first case.
\end{proof}

\begin{remark}
\label{rem:all_sn}
To get the full assertion of Theorem \ref{thm:loc_cyclic}, we required extensions of $\mathbb{Q}(t)$ given by a polynomial $f\in \zz[t,X]$ whose discriminant (as an element of $\zz[t]$) has no non-trivial constant divisor $d\in \zz$. This restriction meant that we could not reduce (as in Theorem \ref{thm:crit}) without loss to a situation in which Theorem \ref{thm:greentao} is applicable, having to resort instead to Theorem \ref{thm:taoziegler}. In particular, since Theorem \ref{thm:taoziegler} does not generalize to arbitrary positive density sets of primes, it is not trivial to generalize our construction from $S_5$ to arbitrary $S_n$, even though our choice of polynomial generalizes very naturally to the $S_n$-polynomial $X^n-t(X-1)$. In trying to do so, the reader may verify that conditions i) to iii) in the proof would be replaced by the following requirements:

\begin{itemize}
\item[i)] $t$ is a prime congruent $1$ modulo $n$,
\item[ii)] $s$ is a prime congruent $1$ modulo $n-1$,
\item[iii)] $(n-1)^{n-1}t - n^ns$ is a prime which splits completely in the splitting field of the polynomial $1+2X+3X^2 + \dots + (n-1)X^{n-2}$.
\end{itemize}

It would be intriguing to know whether an analog of Green-Tao holds for this configuration. This leads to several interesting problems, one of which is the question what the Galois group of $1+2X+3X^2 + \dots + (n-1)X^{n-2}$ is. If it were a group containing a large abelian quotient, then the ``completely split" condition could yield modulo-restrictions which might be incompatible with the representation $(n-1)^{n-1}t - n^ns$. From numerous computer calculations, it however seems that this Galois group is always $S_{n-2}$ or $A_{n-2}$ (except for $n=8$, where it is $PGL_2(5)$!), and the only abelian subextension generated by the square-root of the discriminant of our polynomial would at least pose no local obstruction. The strict verification of the Galois group for all $n$ seems a difficult task, although it can be proven for certain infinite families of values $n$.\footnote{See e.g. the mathoverflow discussion at \\ \texttt{https://mathoverflow.net/questions/44844/galois-groups-of-a-family-of-polynomials}.}
\end{remark}

We now present our application to central embedding problems.
Recall that for $n\ge 5$, $S_n$ has two Schur covers of the form $2.S_n$ - one in which the class of transpositions in $S_n$ is split, and one in which it is not.

\begin{theorem}
\label{thm:centr_ext}
Let $G=S_5$ and let $\Gamma = Z.S_5$ be any central extension of $G$ in which the conjugacy class of transpositions is split. Then there are infinitely many pairwise linearly disjoint locally cyclic $\Gamma$-extensions of $\mathbb{Q}$.
%
\end{theorem}
\begin{proof}
{\it First step: Basic setup}.\\
It is well-known (e.g., \cite[Chapter IV, Theorem 5.1]{MM}) that every central extension can be decomposed into a so-called central {\it Frattini} extension (i.e., with kernel contained in the Frattini subgroup), followed by a split central extension. Since the latter are covered by Lemma \ref{rmk:split}, we may restrict to the former.
Since the Schur multiplier of $S_5$ is of order $2$, as a special case of \cite[Chapter IV, Theorem 6.9]{MM} (and its proof), one obtains that every central Frattini extension of $S_5$ is of the form $\Gamma = Z.S_5$ with $Z$ a cyclic $2$-group. More precisely, one has an exact sequence
$1\to C_2 \to \Gamma \to S_5\times C_{2^k} \to 1$, with $k\in \mathbb{N}_0$.

Let $K/\mathbb{Q}$ be an $S_5$-extension as constructed in the proof of Theorem \ref{thm:loc_cyclic}. Recall that, by construction, $K/\mathbb{Q}$ has the following properties:
\begin{itemize}
\item[1)] Apart from infinity, exactly three primes $p_1, p_2, p_3$ of $\mathbb{Q}$ ramify in $K/\mathbb{Q}$.
\item[2)] The inertia group generators at $p_1$, $p_2$ and $p_3$ are (in this order) a $5$-cycle, a $4$-cycle and a transposition, and the decomposition groups are cyclic of order $5$, $4$ and dividing $6$.
\item[3)] $p_2\equiv 1$ mod $8$. 
\end{itemize}

We claim that we may additionally demand 
\begin{itemize}
\item[4)] Complex conjugation in $K/\mathbb{Q}$ acts as a transposition in $S_5$.
\end{itemize}
Indeed, the extension $K/\mathbb{Q}$ is constructed as the splitting field of a polynomial $X^5 - \frac{t_0}{s_0}(X-1)$, with the only requirement that certain linear forms in two variables become simultaneously prime at $(t_0,s_0)$. Due to Remark \ref{rmk:gt}, one may choose $\frac{t_0}{s_0}$ arbitrarily large and positive, and it is then elementary to verify that the above polynomial has exactly three real roots.

Next, let $L/\mathbb{Q}$ be a $C_{2^k}$-extension with the following properties:
\begin{itemize}
\item[i)] All primes ramifying in $K/\mathbb{Q}$ split completely in $L/\mathbb{Q}$.
\item[ii)] All primes ramifying in $L/\mathbb{Q}$ are completely split in $K(\zeta_{2^{k+1}})/\mathbb{Q}$.
\end{itemize}
Of course, such $L$ exists, e.g., as a very special case of Shafarevich's method for solving split embedding problems with nilpotent kernel (namely, kernel $C_{2^k}$), cf.\ \cite[Theorem 9.6.7]{NSW}, 
and the compositum $LK/\mathbb{Q}$ is then locally cyclic.

{\it Second step: Solving an embedding problem}.\\
We set $H=\Gal(LK/\mathbb{Q})\cong S_5\times C_{2^k}$, fix an epimorphism $\varphi:G_{\mathbb{Q}}\to H$ and consider the embedding problem $(\epsilon, LK/\mathbb{Q})$ given by
\[
  \begin{diagram}\dgARROWLENGTH=1.6em
    \node{}
    \node{}
    \node{}
    \node{G_\mathbb{Q}} \arrow{s,r} {\varphi}
    \node{}\\
    \node{(\epsilon) : 1} \arrow{e}
    \node{C_2} \arrow{e}
    \node{\Gamma} \arrow{e,l} {j}
    \node{H} \arrow{e}
    \node{1}
  \end{diagram}
  \]

We first show that $(\epsilon, LK/\mathbb{Q})$ is solvable. Indeed, for central extensions $2.H$, there is a local-global principle (cf.\ e.g., \cite[Lemma 2.1.5]{Serre}): it suffices to to show solvability of the induced embedding problem over $\mathbb{Q}_p$ for all primes $p$ (including the infinite one). This is due to the Brauer-Hasse-Noether theorem, yielding the injectivity of the map 
$H^2 (G_\qq, C_2) \to \prod_{p} H^2( G_{\qq_p}, C_2)$ (where $G_\qq$ denotes the absolute Galois group of $\qq$, and the product is over all primes, including infinity).
For primes unramified in $KL/\mathbb{Q}$, the induced local embedding problem is automatically solvable. We may therefore restrict to the ramified primes.

Since all those primes split completely in $L/\mathbb{Q}$ by condition i), it suffices to solve the local embedding problems from $\Gal(K/\mathbb{Q})$ to $2.S_5$. First, consider the infinite prime. 
Due to Condition 4), the inertia group of $K/\mathbb{Q}$ at infinity is therefore generated by a transposition, and since the class of transpositions is split in $\Gamma$, the embedding problem over $\mathbb{R}$ is solvable.
As for the three finite primes ramifying in $K/\mathbb{Q}$, $p_1$ has cyclic decomposition group of order $5$, which of course also splits in $\Gamma$. The decomposition group at $p_2$ is cyclic and generated by a $4$-cycle. This class is non-split in $\Gamma$ (i.e., the corresponding elements in $\Gamma$ have order $8$). However, since we have $p_2\equiv 1$ mod $8$, 
 $\mathbb{Q}_{p_2}$ has totally ramified Galois extensions of order $8$, and thus the embedding problem over $\mathbb{Q}_{p_2}$ is solvable as well. Finally, $p_3$ had inertia group generated by a transposition, and cyclic decomposition group of order either $2$ or $6$. Once again, since the class of transpositions splits in $\Gamma$, the preimage of the decomposition group in $\Gamma$ is $C_2\times C_2$ or $C_2\times C_6$, and once again the induced local embedding problem is solvable.

Next, for the primes $p$ ramified in $L/\mathbb{Q}$, due to condition ii) it suffices to solve a local embedding problem of subgroups of $C_{2^{k+1}}$, which is possible due to Kummer theory, since those $p$ were chosen such that $\qq_p$ contains the $2^{k+1}$-th roots of unity.

{\it Third step: Locally cyclic solutions}.\\
We now need to make sure that the solutions to our embedding problem can still be chosen locally cyclic. However, by \cite[Prop.\ 2.1.7]{Serre}, the solution may be chosen in the following way: If $(\psi_p)_p$ is a collection of local solutions at all primes, unramified apart from finitely many $p$, then there is a global solution $\psi$ ramifying only at those $p$ for which $\psi_p$ is ramified (and with the same inertia group). We may therefore choose solutions unramified outside the set of primes ramifying in $KL/\mathbb{Q}$. We then need to ensure cyclic decomposition groups only at those (finite) primes. 
For those primes ramifying in $L/\mathbb{Q}$, this is automatic by condition ii). Consider therefore the three finite primes ramifying in $K/\mathbb{Q}$.
For $p_1$ and $p_2$, nothing needs to be shown, since by construction the decomposition group is contained in a cyclic subgroup of order $10$ and $8$ respectively. For  $p_3$, we need to ensure that the a priori possible decomposition groups $C_2\times C_2$ or $C_2\times C_6$ do not occur, which amounts to ensuring that the primes of $KL/\mathbb{Q}$ extending $p_3$ split completely in the solution field $F$. Assume therefore on the contrary that they remain inert in $F$. We then change the solution field by twisting with a suitable quadratic extension $\mathbb{Q}(\sqrt{q})/\qq$, linearly disjoint from $F/\mathbb{Q}$, such that $p_3$ remains inert in $\mathbb{Q}(\sqrt{q})$. Then $F(\sqrt{q})/\qq$ has Galois group $\Gamma \times C_2$, and the fixed field $F'$ of the diagonal subgroup $C_2$ has $\Gal(F'/\mathbb{Q}) = \Gamma$, with $K\subset F'$. Furthermore by construction, the primes extending $p_3$ in $KL$ split in $F'$, and due to the above arguments, the decomposition groups in $F'/\mathbb{Q}$ at the further ramified primes of $K/\mathbb{Q}$ are still cyclic. It remains to control the prime $q$ itself, since by construction, the inertia group at $q$ in $F'/\mathbb{Q}$ is the central subgroup $C_2$ of $\Gamma$. We are done if we can ensure that $q$ was totally split in $KL$ (since then its decomposition group in $F'/\mathbb{Q}$ is cyclic of order $2$). So our requirements amount to
$$\left(\frac{q}{p_3}\right) = -1, \ Frob_q(KL/\mathbb{Q}) = 1.$$
Assuming additionally $q\equiv 1$ mod $4$, we obtain the sufficient condition $Frob_q(KL(i)/\mathbb{Q}) = 1$ and $Frob_q(\mathbb{Q}(\sqrt{p_3})) \ne 1$. This can clearly be fulfilled as long as $\mathbb{Q}(\sqrt{p_3})$ and $KL(i)$ are linearly disjoint over $\mathbb{Q}$. The latter is indeed the case by the construction of $K$ and $L$; namely, in the quadratic subextension of $K$, not only $p_3$, but also a further prime $p_2 (\equiv 1$ mod $8$) ramifies, whence $\sqrt{p_3}\notin K(i)$. In particular, $KL(i)$ has no quadratic subextension in which $p_3$ is the only odd ramified prime. This concludes the proof.
\end{proof}

\begin{remark}
\begin{itemize}
\item[a)]
Unfortunately, an analogous result for the second Schur cover $2.S_5$ (the one in which the class of transpositions does not split) cannot be deduced from the construction in Theorem \ref{thm:loc_cyclic}. Indeed, that group has only one involution (namely the central one), which enforces that, in order for the embedding problem to be solvable, the corresponding $S_5$-extension must be totally real. However, since the extensions in Theorem \ref{thm:loc_cyclic} were constructed as specializations of function field extensions with three branch points, they cannot be totally real (see, e.g., \cite[Example I.10.2]{MM}).
\item[b)] We leave it to the interested reader to verify that an analog of Theorem \ref{thm:centr_ext} also holds for $G=PGL_2(7)$, as long as the class of involutions in $G\setminus PSL_2(7)$ is split in the central extension.
\end{itemize}
\end{remark}

\section{Appendix: A function field analog}
\label{sec:app}
In this section, we give evidence for a positive answer to Question \ref{ques:cyclic} by showing an analog over global function fields.
\begin{theorem}
\label{thm:functionfields}
Let $G$ be a finite group.
Then there exists $N\in \nn$ such that, for all finite fields $\mathbb{F}_q$ with $char(\mathbb{F}_q)\ge N$, there exist infinitely many pairwise linearly disjoint $\mathbb{F}_q$-regular $G$-extensions $E/\mathbb{F}_q(t)$, 
fulfilling the following:
\begin{itemize}
\item[i)] $E/\mathbb{F}_q(t)$ is tamely ramified.
\item[ii)] At all ramified primes of $E/\mathbb{F}_q(t)$, the decomposition groups are cyclic, equal to the inertia subgroup.
\end{itemize}
\end{theorem}

The proof of Theorem \ref{thm:functionfields} uses a certain amount of machinery from geometric Galois theory. All the ingredients are well-known at this point, although an indepth recollection would lead us too far. For the interested reader, we give references for each of the occurring notions or tools.

\begin{proof}
 Let $q=p^d$ be a prime power and $k_q\supseteq \mathbb{Q}_p$ a local field with residue field $\mathbb{F}_q$.
Let $x_1, \dots, x_n$ be all the non-identity elements of $G$, and let $C_1, \dots, C_n$ be their conjugacy classes in $G$.

Choose a class tuple $\underline{C}$ of $G$ containing all classes $C_i$ sufficiently often and an equal even number of times. Denote by $\mathcal{H}$ the (inner) Hurwitz space of the class tuple $G$. This is a moduli space parameterizing $G$-covers with ramification type $\underline{C}$. By assumption (and by the Riemann existence theorem), $\mathcal{H}$ is non-empty, since if $x_i\in C_i$, then $(x_1, x_1^{-1}, \cdots, x_n, x_n^{-1})$ is a tuple in $\underline{C}$ which generates $G$ and has product $1$, i.e. (by Riemann's existence theorem) corresponds to a $\mathbb{C}$-point of $\mathcal{H}$.
Furthermore, by our assumption the class tuple $\underline{C}$ is $\mathbb{Q}$-rational (see \cite[Section I.4.2]{MM} for a definition). This implies that the Hurwitz space $\mathcal{H}$ is a variety defined over $\mathbb{Q}$; denote the function field of this Hurwitz space by $\mathbb{Q}(\mathcal{H})$.
Assume now that $Z(G)=1$ and that the Schur multiplier of $G$ is generated by commutators. Then the Conway-Parker theorem ensures that, since $\underline{C}$ contains every non-trivial class of $G$ sufficiently often, the Hurwitz space $\mathcal{H}$ is an {\it absolutely irreducible variety}. Note also that the above assumptions on $G$ are without loss, since every finite group is a quotient of a group with these properties. See, e.g., \cite[Chapter III, Theorem 6.10]{MM} for the Conway-Parker theorem.

Following, e.g., \cite{FD}, one then has the existence of a  ``universal family", that is, a $\mathbb{Q}(\mathcal{H})$-regular $G$-extension $U / \mathbb{Q}(\mathcal{H})(t)$, with ramification given by the above class tuple, and such that, for every field $K\supseteq \mathbb{Q}$ and every $K$-rational point $P$ on $\mathcal{H}$, constant reduction of $U\cdot K / K(\mathcal{H})(t)$ at $P$ yields a $K$-regular $G$-extension of $K(t)$ (this is due to Fried and V\"olklein, see \cite{FV}).

Furthermore, it is well known (under the name ``Half Riemann existence theorem") that $\mathcal{H}$ as above possesses infinitely many
rational points over the complete field $k_q$, or equivalently, there are infinitely many $k_q$-regular $G$-extensions of $k_q(t)$, with monodromy 
$(x_1,x_1^{-1}, \dots ,x_n,x_n^{-1})$.
More precisely, this result can be obtained using the technique of {\it algebraic patching}, as outlined in detail, e.g., in \cite{Jarden}. The construction, briefly, is as follows: Begin with $k_q$-regular extensions $E_1/k_q(t),\dots, E_n/k_q(t)$ with cyclic Galois group $\langle x_i\rangle$ and such that all ramified places are totally ramified (obtaining such extensions is easy, e.g., using the rigidity method, and they can in fact even be defined over $\mathbb{Q}$). Then, assuming that the ramification loci of the $E_i/k_q(t)$ are sufficiently separated (in the norm of the complete field $k_q$), one obtains the existence of a $k_q$-regular $G$-extension $E/k_q(t)$ whose branch point set is exactly the union of the branch point sets of the $E_i/k_q(t)$, and whose local behavior (in particular, whose inertia and decomposition groups) at each branch point is exactly the same as the behavior at the corresponding branch point in $E_i/k_q(t)$. Since we assumed that all ramified places are totally ramified in the $E_i/k_q(t)$, this implies that all non-trivial inertia groups in $E/k_q(t)$ equal the corresponding decomposition group. 

Now what is left to do is to obtain a ``good reduction" argument, to reduce $E_i/k_q(t)$ to an extension $\widehat{E_i}/\mathbb{F}_q(t)$ with the same Galois group and same local behavior at the branch points. Since the construction in the previous step depended on $q$, it may not be obvious immediately that this is possible.

Consider however the compositum of all residue extensions at branch points in $U/\mathbb{Q}(\mathcal{H})(t)$; this is a finite extension $\mathbb{Q}(\mathcal{H}')/\mathbb{Q}(\mathcal{H})$, corresponding to a morphism $\mathcal{H}'\to \mathcal{H}$ of varieties. From the definition of $\mathcal{H'}$, for any field $K\supseteq \mathbb{Q}$, the $K$-rational points of $\mathcal{H}'$ parameterize Galois covers with group $G$ (and ramification type as chosen above) with trivial residue extensions at all branch points. 
Since we already know from the above constructions that there are points on $\mathcal{H}$ at which all residue extensions at branch points become trivial over $k_q$, this means that there are $k_q$-rational points on $\mathbb{Q}(\mathcal{H}')$ for all $q$.

In particular, $\mathbb{Q}(\mathcal{H}')/\mathbb{Q}$ must be $\qq$-regular (since if it contained any proper finite extension of $\mathbb{Q}$, there would not be any $\mathbb{Q}_p$-points on $\mathbb{Q}(\mathcal{H}')$ for primes $p$ which remain inert in that constant extension). So $\mathbb{Q}(\mathcal{H}')$ corresponds to a variety $\mathcal{H}'$ which is absolutely irreducible over $\mathbb{Q}$. Therefore, by the Lang-Weil bound, $\mathcal{H}'$ has (many) $\mathbb{F}_p$-rational points, for all sufficiently large primes $p$. Now, we may use good reduction (after exempting finitely many primes $p$): Namely, given a $\mathbb{Q}_p$-point $P$ arising from lifting such an $\mathbb{F}_p$-point, consider the constant reduction of $U\cdot \mathbb{Q}(\mathcal{H}')/\mathbb{Q}(\mathcal{H}')(t)$ at $P$. This yields a $\mathbb{Q}_p$-regular $G$-extension of $\mathbb{Q}_p(t)$ with all decomposition groups at branch points cyclic, and mod-$p$ reduction of this extension yields the same result over $\mathbb{F}_p(t)$. See, e.g., \cite{Wew98} for good reduction of Hurwitz spaces.

Lastly, it remains to justify that one may obtain infinitely many linearly disjoint extensions of $\mathbb{F}_q(t)$ in this way.
To that aim, choose a sequence $(\underline{C}_n)_{n\in \mathbb{N}}$ of class tuples where $\underline{C}_n$ is the $n$-fold repetition of $\underline{C}=\underline{C}_1$, and let $E_n/\mathbb{F}_q(t)$ be an extension of type $\underline{C}_n$ as obtained above. For every proper normal subgroup $N$ of $G$, choose a class $C_i$ not contained in $N$. Then for $m>n$, more points ramify of class $C_i$ in $E_{m}^N/\mathbb{F}_q(t)$ than in all of $E_n/\mathbb{F}_q(t)$, showing that $E_m$ and $E_n$ are linearly disjoint over $\mathbb{F}_q(t)$.
\end{proof}


\end{document}